\documentclass[12pt,reqno]{amsart}

\usepackage[a4paper,margin=1in]{geometry}
\usepackage{amsmath}
\usepackage{amssymb}
\usepackage{amsthm}
\usepackage{mathtools}
\usepackage{microtype}
\usepackage{indentfirst}
\usepackage[numbers,compress,sort]{natbib}
\usepackage[colorlinks=true,
	linkcolor=blue,
	citecolor=blue,
	urlcolor=blue,
	filecolor=blue,
	menucolor=blue,
	runcolor=blue,
	pdftitle={Chern's Conjecture with Constant Cubic Trace},
	pdfauthor={Huixin Tan, Zizhou Tang, Yuquan Xie, Wenjiao Yan}]{hyperref}

\allowdisplaybreaks[4]
\newtheorem{theorem}{Theorem}[section]
\newtheorem{lemma}[theorem]{Lemma}
\newtheorem{proposition}[theorem]{Proposition}
\newtheorem{corollary}[theorem]{Corollary}
\theoremstyle{definition}

\theoremstyle{remark}
\newtheorem{remark}[theorem]{Remark}

\DeclareMathOperator{\Vol}{Vol}
\DeclareMathOperator{\tr}{tr}
\newcommand{\sphere}{\mathbb S}
\newcommand{\R}{\mathbb R}
\newcommand{\cH}{\mathcal H}
\newcommand{\cV}{\mathcal V}

\title[Chern's Conjecture with Constant Cubic Trace]{Chern's Conjecture with Constant Cubic Trace}

\author{Huixin Tan}
\address{School of Mathematical Sciences, Laboratory of Mathematics and Complex Systems, Beijing Normal University, Beijing 100875, P. R. China}
\email{hxtan@mail.bnu.edu.cn}

\author{Zizhou Tang}
\address{Chern Institute of Mathematics \& LPMC, Nankai University, Tianjin 300071, P. R. China}
\email{zztang@nankai.edu.cn}

\author{Yuquan Xie}
\address{School of Mathematics, Hangzhou Normal University, Hangzhou 311121, P. R. China}
\email{yuqxie@hznu.edu.cn}

\author{Wenjiao Yan}
\address{School of Mathematical Sciences, Laboratory of Mathematics and Complex Systems, Beijing Normal University, Beijing 100875, P. R. China}
\email{wjyan@bnu.edu.cn}

\thanks{The project is partially supported by the NSFC (No.12271038, 12371048, 12526205), Nankai Zhide Foundation, and by the Open Project of the Key Laboratory of Mathematics and Complex Systems, Beijing Normal University (No. K202503).}

\begin{document}\raggedbottom

\begin{abstract}
	We prove that the values set of \(S=|A|^2\) attained by closed embedded minimal hypersurfaces in \(\mathbb S^{n+1}(1)\) with constant \(S\) and constant \(f_3=\operatorname{tr}(A^3)\) is locally finite, where \(A\) denotes the shape operator. Neither the topology of the hypersurface nor the value of \(f_3\) is fixed.
\end{abstract}

\maketitle

\section{Introduction}\label{sec:introduction}

Let \(F:M^n\to\sphere^{n+1}(1)\) be a closed minimal hypersurface, let \(A\) be its shape operator, and put \(S=\tr(A^2)\).
The Gauss equation gives the scalar curvature identity
\[
	R=n(n-1)-S,
\]
so the constant scalar curvature condition is equivalent to the constancy of \(S\).

\textbf{S. S. Chern's conjecture} asks, with the dimension \(n\) fixed but with the closed source manifold and the immersion otherwise arbitrary, whether the possible constant values of \(S\) form a discrete set \cite{Chern1968,ChernDoCarmoKobayashi1970}.
The same formulation appears as Problem~105 in \textbf{Yau's problem} list \cite[p.~693]{Yau1982}.

Simons' rigidity inequality implies that a closed minimal hypersurface satisfying \(0\leq S\leq n\) has \(S=0\) or \(S=n\) \cite{Simons1968}.
The equality case \(S=n\) was identified independently by Chern, do Carmo, and Kobayashi and by Lawson as a Clifford minimal hypersurface \cite{ChernDoCarmoKobayashi1970,Lawson1969}.
Peng and Terng derived higher-order integral inequalities and used them to prove the next pinching gap above \(n\): if the constant \(S\) is greater than \(n\), then \(S>n+\frac{1}{12n}\) \cite{PengTerng1983a,PengTerng1983b}.
Their work initiated a long sequence of pinching results.
Among later improvements are the estimates of Ding and Xin and of Xu and his students \cite{DingXin2011,XuXu2017,LeiXuXu2021}.
Building on Peng and Terng's three-dimensional estimates, Chang completed the proof of Chern's conjecture for \(n=3\) \cite{Chang1993}.
Much of the subsequent work has concentrated on the first and second gap problems, which exclude new constant values of \(S\) near the known models.
Such pinching results isolate the first few admissible values but do not control the full value set.
Apart from Chang's work \cite{Chang1993} in dimension three and classification results under additional hypotheses, the discreteness problem remains open.

A connected hypersurface in the sphere is called \emph{isoparametric} if all of its principal curvatures are constant.
Isoparametric minimal hypersurfaces supply the principal model family for Chern's conjecture.
If \(g\) denotes the number of distinct principal curvatures, Cartan's work and M\"unzner's structure theorems show that \(g\) is restricted to \(1,2,3,4,\) or \(6\), and a minimal member has \(S=(g-1)n\) \cite{Cartan1939,Munzner1980,Munzner1981}.
These examples therefore yield the familiar values \(0,n,2n,3n,\) and \(5n\) whenever the corresponding multiplicities exist.
The stronger classification expectation that every closed minimal hypersurface with constant scalar curvature is isoparametric remains open in general.
The surveys of Ge and Tang and of Scherfner, Weiss, and Yau give broader accounts of the conjecture and its relation to isoparametric geometry \cite{GeTang2012,ScherfnerWeissYau2012}.

Although Chern's conjecture is formulated for immersed hypersurfaces, the embedded subclass is natural. It contains all of the standard isoparametric models, excludes covering multiplicities, and gives a global separation of the sphere.
Historically, however, imposing embeddedness has led to little additional progress: the principal gap and rigidity results recalled above already hold for immersions, and the conjecture remains open in general even in the embedded category.

Additional symmetric functions of the principal curvatures have proved useful in attacking this classification problem.
De Almeida and Brito proved that a closed hypersurface in \(\sphere^4\) with constant mean curvature and constant nonnegative scalar curvature is isoparametric \cite{AlmeidaBrito1990}.
Tang, Wei, and Yan extended this result to \(n>3\) under \(g=n\) everywhere and \(\int_M R d \mu\geq0\): constancy of \(f_k=\tr(A^k)\) for \(1\leq k\leq n-1\) implies isoparametricity \cite{TangWeiYan2020}.
Tang and Yan removed the condition \(g=n\), proving the same conclusion for \(n\geq4\) under \(R\geq0\) and constant \(f_k\), \(1\leq k\leq n-1\) \cite{TangYan2023}.
For constant \(f_3\), Cheng, Wei, and Yamashiro obtained a dimension-independent pinching theorem \cite{ChengWeiYamashiro2025}.
In dimension four, Deng, Gu, and Wei proved isoparametricity for closed minimal hypersurfaces with constant scalar curvature satisfying the Willmore condition \cite{DengGuWei2017}.
These results are classification or gap theorems, whereas the conclusion below concerns local finiteness of the value set in every dimension.

We now formulate the class precisely.
For \(n\geq2\), let $\mathcal{E}_n^{(3)}$ be the set of numbers \(s\geq0\) for which there exist a closed connected smooth \(n\)-manifold \(M\) and a smooth embedding \(F:M\hookrightarrow\sphere^{n+1}(1)\) such that, with respect to a global unit normal,
\[
	H=0,
	\qquad
	S\equiv s,
	\qquad
	f_3\equiv q
\]
for some real number \(q\).

\begin{theorem}\label{thm:main}
	For every integer \(n\geq2\) and every finite \(\Lambda\geq0\), the set
	\[
			\mathcal{E}_n^{(3)}\cap[0,\Lambda]
	\]
	is finite.
\end{theorem}

Theorem~\ref{thm:main} establishes the discreteness predicted by Chern's conjecture in the stronger form of local finiteness for embedded hypersurfaces with constant \(f_3\).
Here the hypersurface is not fixed: its diffeomorphism type, its embedding, and the constant value of \(f_3\) may all vary.

\begin{remark}\label{rem}
	Chern's original conjecture concerns closed minimal submanifolds with arbitrary codimension: for fixed $n, m\geq 1$, the possible constant values of $S$ realized by closed minimally immersed $n$-submanifolds of $\sphere^{n+m}(1)$ are conjectured to form a discrete set. Most subsequent work has focused on its hypersurface case $m=1$.
	Quite recently, Firester and Tsiamis disproved the conjecture in higher codimension by constructing, for $n\geq3$ and \(m\geq4\), and also for even \(n\geq4\) and \(m\geq3\), countable families of closed embedded minimal $n$-submanifolds whose constant $S$-values are dense in a bounded interval \cite{FiresterTsiamis2026}.
	Thus embeddedness alone does not enforce discreteness in higher codimension.
    In codimension one, however, embeddedness together with a uniform bound on \(S\) yields the two-sided tubular control needed for compactness.
    Together with the constancy of $f_3$, this leads to Theorem \ref{thm:main} and provides further evidence for the hypersurface conjecture.
\end{remark}

The proof of the main theorem consists of a local part and a global part.
For the local part, consider a \(C^1\) family \(F_t\) with values in \(C^3(M,\sphere^{n+1}(1))\), such that every \(F_t\) is minimal and both \(S_t\) and \(f_{3,t}\) are spatially constant. If \(\phi\) is the normal speed, the linearized mean curvature equation and the variation of \(S\) give
\[
	0=H'=\Delta\phi+(n+S)\phi,
	\qquad
	\frac12S'=\langle A,\nabla^2\phi\rangle+f_3\phi.
\]
Integration, the Codazzi equation, and the constancy of \(f_3\) imply \(S'=0\).
This path rigidity does not by itself control a possibly singular moduli space, so we reduce the mean-curvature equation to a finite-dimensional real analytic zero set and impose constancy of \(S\) and \(f_3\) by analytic variance functionals.
The real analytic curve-selection lemma then converts any hypothetical local accumulation of \(S\)-values into an analytic path, contradicting path rigidity.

For the global part, a uniform bound for \(S\) gives a uniform bound for the principal curvatures but does not by itself give an area bound or multiplicity-one compactness.
Embeddedness supplies the missing global control.
Each hypersurface separates the sphere into two domains, and Howard's rigidity form of the rolling theorem, together with the strict positivity of the Ricci tensor, yields equality of the rolling and focal radii on each side \cite{Howard1999}.
The focal formula in the sphere then yields a uniform two-sided normal injectivity radius.
Integrating the normal Jacobian over a shorter tube gives a uniform area bound, after which Hausdorff compactness, the uniform reach bound, varifold compactness, and elliptic regularity give smooth embedded multiplicity-one compactness.
The local result at a smooth limit then completes the proof of Theorem~\ref{thm:main}.

Section~\ref{sec:variation} fixes the geometric conventions and derives the variation formulas used later.
Section~\ref{sec:path-rigidity} proves a smooth path rigidity.
Section~\ref{sec:local-rigidity} proves local constancy of the \(S\)-value on the finite-dimensional analytic set defined by the minimality and constancy conditions.
Section~\ref{sec:compactness} proves the uniform tube, area, and compactness statements.
Section~\ref{sec:global-proof} proves Theorem~\ref{thm:main} and discusses its precise scope.

\section{Conventions and variation formulas}\label{sec:variation}

Throughout the paper, \(M^n\) is closed and connected, the ambient sphere has sectional curvature one, and \(F:M\to\sphere^{n+1}(1)\) is a two-sided hypersurface, meaning that its normal line bundle is trivial. We fix a global unit normal \(\nu\).
Since \(\sphere^{n+1}(1)\) is orientable, two-sidedness is equivalent here to the orientability of \(M\). For the closed embedded hypersurfaces considered in the main theorem, two-sidedness is automatic because they separate the sphere.
Let \(\overline g\) and \(\overline\nabla\) denote the round metric on \(\sphere^{n+1}(1)\) and its Levi-Civita connection, respectively. The induced metric on \(M\) is \(g=F^*\overline g\), and its Levi-Civita connection is denoted by \(\nabla\). We use the shape-operator convention
\[
	A(X)=-\bigl(\overline\nabla_X\nu\bigr)^{\top},
\]
where \((\cdot)^{\top}\) denotes orthogonal projection onto \(TM\). We write \(h(X,Y)=\langle AX,Y\rangle\), \(H=\tr A\), \(S=\tr(A^2)\), and \(f_3=\tr(A^3)\). The Laplacian is \(\Delta=\tr\nabla^2\), so its spectrum on a closed manifold is nonpositive.

Since the unit sphere is Einstein, its ambient Ricci tensor satisfies \(\overline{\operatorname{Ric}}=n\overline g\), and hence \(\overline{\operatorname{Ric}}(\nu,X)=0\) for every \(X\in TM\).
Thus the contracted Codazzi equation reduces to
\begin{equation}\label{eq:divA}
	\operatorname{div}A=\nabla H.
\end{equation}
In particular, \(\operatorname{div}A=0\) on a minimal hypersurface.
The Gauss equation gives
\[
	R=n(n-1)+H^2-S,
\]
and hence \(R=n(n-1)-S\) when \(H=0\).

\begin{lemma}[Normal variation formulas]\label{lem:variation}
	Let \(I\subset\R\) be an interval containing \(0\), and let \(t\mapsto F_t\) be a \(C^1\) map from \(I\) into \(C^3(M,\sphere^{n+1}(1))\).
	Assume that each \(F_t\) is a two-sided immersion and that \(F_0=F\).
	Choose the global unit normals \(\nu_t\) to depend \(C^1\) on \(t\), with \(\nu_0=\nu\), and suppose that \(\left.\partial_tF_t\right|_{t=0}=\phi\nu\).
	Let \(g_t\) and \(A_t\) denote the induced metric and shape operator.
	Then
	\begin{align}
		g'&=-2\phi h,\label{eq:gvariation}\\
		A'&=\nabla^2\phi+\phi(A^2+I),\label{eq:Avariation}
	\end{align}
	where the Hessian in \eqref{eq:Avariation} is viewed as a self-adjoint endomorphism.
	If \(F\) is minimal, then
	\begin{align}
		H'&=\Delta\phi+(n+S)\phi,\label{eq:Hvariation}\\
		S'&=2\bigl(\langle A,\nabla^2\phi\rangle+f_3\phi\bigr),\label{eq:Svariation}
	\end{align}
    where, in a local $g$-orthonormal frame $\{e_i\}$, $h_{ij}=h(e_i,e_j)$, $\phi_{ij}=\nabla^2\phi(e_i,e_j)$, and $\langle A,\nabla^2\phi\rangle=\sum_{i,j}h_{ij}\phi_{ij}$. Here primes denote derivatives at \(t=0\).
\end{lemma}

\begin{proof}
	Let \(X,Y\) be vector fields on \(M\), extended independently of \(t\), and identify \(dF_0(X)\) and \(dF_0(Y)\) with \(X\) and \(Y\). For a vector field \(Z_t\) along \(F_t\), write
	\(\mathcal D_tZ_t=\overline\nabla_{\partial_tF_t}Z_t\).
	All quantities below are evaluated at \(t=0\).

	Put \(V=\partial_tF_t|_{t=0}=\phi\nu\). Since \([\partial_t,X]=0\), the torsion-free property of the ambient connection gives
	\[
		\left.\mathcal D_t\bigl(dF_t(X)\bigr)\right|_{t=0}
		=\overline\nabla_XV
		=X(\phi)\nu-\phi AX.
	\]
	Differentiating \(g_t(X,Y)=\langle dF_t(X),dF_t(Y)\rangle\) and using the preceding identity, we obtain
	\[
		\begin{aligned}
			g'(X,Y)
			&=\langle X(\phi)\nu-\phi AX,Y\rangle
			  +\langle X,Y(\phi)\nu-\phi AY\rangle  \\
			&=-2\phi h(X,Y).
		\end{aligned}
	\]
	This proves \eqref{eq:gvariation}.

	We next determine the variation of the unit normal. Differentiating
	\(\langle\nu_t,\nu_t\rangle=1\) shows that
	\(\mathcal D_t\nu_t|_{t=0}\) is orthogonal to \(\nu\). Differentiating
	\(\langle\nu_t,dF_t(X)\rangle=0\) gives
	\[
		\left\langle\left.\mathcal D_t\nu_t\right|_{t=0},X\right\rangle
		=-\left\langle\nu,
		\left.\mathcal D_t\bigl(dF_t(X)\bigr)\right|_{t=0}\right\rangle
		=-X(\phi).
	\]
    The vector \(\left.\mathcal D_t\nu_t\right|_{t=0}\) is tangent to the ambient sphere and, by the preceding normalization identity, is orthogonal to \(\nu\).
    Since \(T_{F(x)}\sphere^{n+1}(1)=dF(T_xM)\oplus\mathbb R\nu\), it therefore belongs to \(dF(T_xM)\). The preceding identity then determines it uniquely as
	\[
		\left.\mathcal D_t\nu_t\right|_{t=0}=-\nabla\phi.
	\]

	Consider now the covariant second fundamental form
	\(h_t(X,Y)=\langle\overline\nabla_{dF_t(X)}dF_t(Y),\nu_t\rangle\).
	Differentiation yields
	\[
		h'(X,Y)
		=
		\left\langle
		\left.\mathcal D_t
		\bigl(\overline\nabla_{dF_t(X)}dF_t(Y)\bigr)\right|_{t=0},
		\nu
		\right\rangle
		+
		\left\langle\overline\nabla_XY,
		\left.\mathcal D_t\nu_t\right|_{t=0}\right\rangle .
	\]
	The curvature commutation formula and \([\partial_t,X]=0\) give
	\[
		\left.\mathcal D_t
		\bigl(\overline\nabla_{dF_t(X)}dF_t(Y)\bigr)\right|_{t=0}
		=
		\overline\nabla_X\bigl(Y(\phi)\nu-\phi AY\bigr)
		+\overline R(\phi\nu,X)Y.
	\]
	The normal component of the first term on the right is
	\(X(Y(\phi))-\phi h(X,AY)\). For the unit sphere,
	\(\overline R(U,V)W=\langle V,W\rangle U-\langle U,W\rangle V\), and therefore
	\(\langle\overline R(\phi\nu,X)Y,\nu\rangle=\phi g(X,Y)\).
	The variation of the normal contributes
	\[
		\left\langle\overline\nabla_XY,
		\left.\mathcal D_t\nu_t\right|_{t=0}\right\rangle
		=-\langle\nabla_XY,\nabla\phi\rangle.
	\]
	Combining these identities and using
	\(\nabla^2\phi(X,Y)=X(Y(\phi))-\langle\nabla_XY,\nabla\phi\rangle\), we find
	\[
		h'(X,Y)
		=\nabla^2\phi(X,Y)+\phi g(X,Y)-\phi h(X,AY).
	\]
	Equivalently,
	\[
		h'=\nabla^2\phi+\phi g-\phi h\circ A,
		\qquad
		(h\circ A)(X,Y)=h(X,AY)=\langle A^2X,Y\rangle.
	\]

	It remains to pass from the covariant tensor \(h\) to the shape operator \(A\). Differentiating \(h_t(X,Y)=g_t(A_tX,Y)\) gives
	\(h'(X,Y)=g'(AX,Y)+g(A'X,Y)\). Hence, by \eqref{eq:gvariation} and the formula for \(h'\),
	\[
		\begin{aligned}
			g(A'X,Y)
			&=h'(X,Y)-g'(AX,Y) \\
			&=\nabla^2\phi(X,Y)+\phi g(X,Y)
			  -\phi\langle A^2X,Y\rangle
			  +2\phi\langle A^2X,Y\rangle \\
			&=\nabla^2\phi(X,Y)+\phi g(X,Y)
			  +\phi\langle A^2X,Y\rangle.
		\end{aligned}
	\]
	Viewing the Hessian as a self-adjoint endomorphism therefore gives
	\[
		A'=\nabla^2\phi+\phi(A^2+I),
	\]
	which proves \eqref{eq:Avariation}. Notice that the term
	\(2\phi A^2\) arising from the variation of the metric is precisely what changes the term \(-\phi A^2\) in \(h'\) into \(+\phi A^2\) in \(A'\).

	Finally, since \(H=\tr A\), taking the trace of \eqref{eq:Avariation} gives
	\[
		H'=\Delta\phi+(n+S)\phi.
	\]
	This identity does not require minimality. Similarly, regarding \(A_t\) as a family of endomorphisms of \(TM\), we obtain
	\[
		\begin{aligned}
			S'
			&=\tr\bigl((A^2)'\bigr)
			  =2\tr(AA') \\
			&=2\langle A,\nabla^2\phi\rangle
			  +2\phi\tr(A^3)+2\phi\tr A \\
			&=2\langle A,\nabla^2\phi\rangle+2\phi f_3+2\phi H.
		\end{aligned}
	\]
	When \(F\) is minimal, \(H=0\), and the last identity reduces to
	\[
		S'=2\bigl(\langle A,\nabla^2\phi\rangle+f_3\phi\bigr).
	\]
	This proves \eqref{eq:Hvariation} and \eqref{eq:Svariation}.
\end{proof}

\begin{remark}\label{rem:tangential}
	For a general variation field \(W+\phi\nu\), the right-hand sides of the scalar variation formulas acquire the transport terms \(W(H)\) and \(W(S)\).
	These terms vanish along a family on which \(H\) is zero and \(S\) is spatially constant.
	Equivalently, a time-dependent reparametrization removes \(W\), so Lemma~\ref{lem:variation} applies to the geometric family without changing its images.
\end{remark}

\section{Rigidity along differentiable families}\label{sec:path-rigidity}

The first ingredient is an exact no-drift statement for the constant \(S\)-value along a constrained family.
To specify the regularity of such a family, we regard \(C^3(M,\sphere^{n+1}(1))\) as the Banach manifold of \(C^3\) maps from \(M\) into \(\sphere^{n+1}(1)\), with the topology induced by \(C^3(M,\mathbb R^{n+2})\).
A path \(t\mapsto F_t\) is \(C^1\) in this space if, after viewing \(F_t\) as an \(\mathbb R^{n+2}\)-valued map, the derivative \(\partial_tF_t\) exists in \(C^3(M,\mathbb R^{n+2})\) and depends continuously on \(t\) in the \(C^3\) norm.
The resulting rigidity statement does not require the maps \(F_t\) to be embeddings.

\begin{proposition}\label{prop:path-rigidity}
	Let \(I\subset\R\) be an interval, and let \(t\mapsto F_t\) be a \(C^1\) map from \(I\) into \(C^3(M,\sphere^{n+1}(1))\).
	Assume that every \(F_t\) is a two-sided minimal immersion and that both \(S_t\) and \(f_{3,t}\) are spatially constant on \(M\).
	Then the function \(t\mapsto S_t\) is constant on \(I\).
\end{proposition}

\begin{proof}
    Fix \(t\in I\). On a parameter neighborhood of \(t\), choose the global unit normals to depend \(C^1\)-smoothly on the parameter. Write
    \[
        \partial_tF_t=dF_t(V)+\phi\nu=:W+\phi\nu,
    \]
    where \(V\in\Gamma(TM)\), \(W=dF_t(V)\) is the tangential component, and \(\phi=\langle\partial_tF_t,\nu\rangle\) is the normal speed. The tangential component changes only the parametrization, so Remark~\ref{rem:tangential} allows us to discard its transport terms. In the calculation below, all geometric quantities and the measure \( d \mu\) are those of the fixed member \(F_t\); the subscript \(t\) is suppressed.
	Because the family is minimal, Equation~\eqref{eq:Hvariation} gives
	\begin{equation}\label{eq:jacobi-family}
		\Delta\phi+(n+S)\phi=0.
	\end{equation}
	Integrating \eqref{eq:jacobi-family} over the closed hypersurface yields
    \begin{equation}\label{eq:mean-zero-speed}
        \int_M\phi d \mu=0,
    \end{equation}
    since \(S\) is spatially constant, \(S\geq0\), and hence \(n+S\) is a positive constant that may be taken outside the integral.
	Since \(S_t\) is spatially constant, its derivative is a real number, and Equation~\eqref{eq:Svariation} may be written as
	\begin{equation}\label{eq:c-definition}
		\frac12\frac{d S_t}{d t}=\langle A,\nabla^2\phi\rangle+f_3\phi=:c(t).
	\end{equation}
	The contracted Codazzi identity \eqref{eq:divA} and minimality give
	\begin{equation}\label{eq:hessian-integral-zero}
		\int_M\langle A,\nabla^2\phi\rangle d \mu
		=-\int_M\langle\operatorname{div}A,\nabla\phi\rangle d \mu
		=0.
	\end{equation}
	Integrating \eqref{eq:c-definition}, using \eqref{eq:hessian-integral-zero}, the spatial constancy of \(f_3\), and \eqref{eq:mean-zero-speed} gives
	\[
		c(t)\Vol(M)=f_3\int_M\phi d \mu=0.
	\]
	Thus \(c(t)=0\) at every \(t\), and hence \(S_t\) is constant on \(I\).
\end{proof}

\begin{remark}\label{rem:f3-essential}
	Without the spatial constancy of \(f_3\), the same calculation gives
	\[
		\frac12S'\Vol(M)=\int_M\phi f_3 d \mu,
	\]
	and neither the Jacobi equation nor \(\int_M\phi d \mu=0\) forces the right-hand side to vanish.
	This is the precise point at which the constant-\(f_3\) hypothesis enters the local argument.
\end{remark}

\begin{remark}\label{rem:range-condition}
	The preceding observation has a natural Fredholm formulation.
	Fix \(t\in I\), and let \(L_t=\Delta_t+n+S_t\) be the Jacobi operator of \(F_t\).
	By \eqref{eq:jacobi-family}, the normal speed \(\phi\) belongs to \(\ker L_t\).
	Since \(L_t\) is self-adjoint and Fredholm on the closed manifold \(M\), the Fredholm alternative gives
	\[
		f_{3,t}\perp_{L^2(M,g_t)}\ker L_t
		\quad\Longleftrightarrow\quad
		f_{3,t}\in\operatorname{Ran}L_t
		=
		\operatorname{Ran}(\Delta_t+n+S_t).
	\]
	Thus the weaker range condition already forces the integral in Remark~\ref{rem:f3-essential} to vanish and yields \(\frac{d S_t}{d t}=0\).
	If \(f_{3,t}\) is spatially constant, then \(f_{3,t}=L_t\bigl(f_{3,t}/(n+S_t)\bigr)\), so the hypothesis of Proposition~\ref{prop:path-rigidity} implies the range condition.
	Consequently, the constant-\(f_3\) assumption in Theorem~\ref{thm:main} is a convenient sufficient condition for path rigidity, whereas the range condition above provides a weaker sufficient condition expressed in spectral terms.
	Extending Theorem~\ref{thm:main} to the range condition requires showing, through the finite-dimensional reduction of Section~\ref{sec:local-rigidity}, that this condition defines a semianalytic constrained set.
    The local rigidity statement must then be formulated relative to this set, so that the global compactness argument imposes the range condition on the approximating hypersurfaces without requiring it to pass to the limiting hypersurface.
\end{remark}

\section{Analytic local rigidity}\label{sec:local-rigidity}

We next pass from Proposition~\ref{prop:path-rigidity}, which concerns differentiable paths, to a neighborhood statement without assuming that the local zero set of the mean curvature operator is a manifold.
The finite-dimensional reduction used below is the standard Lyapunov--Schmidt construction for the mean-curvature operator, consistent with the general Fredholm description of minimal submanifolds in \cite{White1991}.

Fix a smooth closed embedded minimal hypersurface \(F:M\hookrightarrow\sphere^{n+1}(1)\), choose a global unit normal \(\nu\), and fix \(k\geq4\) and \(0<\alpha<1\).
Here \(C^{k,\alpha}(M)\) denotes the usual H\"older space of real-valued functions on \(M\), defined with respect to the induced metric \(g\).
For \(u\in C^{k,\alpha}(M)\) sufficiently small, the corresponding spherical normal graph is
\begin{equation}\label{eq:normal-graph}
	F_u(x)=\cos(u(x))F(x)+\sin(u(x))\nu(x).
\end{equation}
For \(u\) sufficiently small, \(F_u\) is an embedding. Up to reparametrization, every nearby embedding is uniquely represented as such a normal graph.
Let \(\pi_F\) be the nearest-point projection from a tubular neighborhood of \(F(M)\) onto \(F(M)\).
If an embedding \(G:M\to\sphere^{n+1}(1)\) is sufficiently close to \(F\) in \(C^{k,\alpha}\), then
\[
	p_G=F^{-1}\circ\pi_F\circ G:M\longrightarrow M
\]
is a \(C^{k,\alpha}\) diffeomorphism. The reparametrized embedding \(G\circ p_G^{-1}\) meets each normal fiber of \(F(M)\) once and hence has a unique representation \(F_u\) of the form \eqref{eq:normal-graph}.
Writing \(\operatorname{Emb}^{k,\alpha}(M,\sphere^{n+1}(1))\) for the space of \(C^{k,\alpha}\) embeddings and \(\operatorname{Diff}^{k,\alpha}(M)\) for the group of \(C^{k,\alpha}\) diffeomorphisms of \(M\), acting on embeddings by precomposition, the small normal graphs form a local slice for \(\operatorname{Emb}^{k,\alpha}(M,\sphere^{n+1}(1))\big/\operatorname{Diff}^{k,\alpha}(M)\).

Let \(\mathcal U\subset C^{k,\alpha}(M)\) be a sufficiently small open neighborhood of the origin.
For \(u\in\mathcal U\), set
\[
	P_u=\cos u\,I-\sin u\,A,
	\qquad
	Q_u=\sin u\,I+\cos u\,A,
	\qquad
	\eta_u=-\sin u\,F+\cos u\,\nu.
\]
Since \(P_0=I\), \(P_u\) is invertible for \(u\) sufficiently small.
Taking \(\mathcal U\) smaller if necessary, we may also assume that \(F_u\) is an embedding for every \(u\in\mathcal U\).
Differentiating \eqref{eq:normal-graph} and using \(d\nu(X)=-AX\) give
\[
	dF_u(X)=P_uX+du(X)\eta_u
\]
for every \(X\in T_xM\).
All gradients and norms below are taken with respect to \(g\).
Set \(\xi_u=P_u^{-1}\nabla u\).
Since \(\xi_u\) is orthogonal to both \(F\) and \(\nu\), the vectors \(\eta_u\) and \(\xi_u\) are tangent to \(\sphere^{n+1}(1)\) at \(F_u(x)\).
Moreover, since \(P_u\) is self-adjoint, we have
\[
	\left\langle \eta_u-\xi_u,dF_u(X)\right\rangle
	=du(X)-\left\langle P_u\xi_u,X\right\rangle
	=du(X)-\left\langle\nabla u,X\right\rangle
	=0.
\]
Since \(\eta_u\perp\xi_u\) and \(|\eta_u|=1\), we have \(|\eta_u-\xi_u|^2=1+|\xi_u|^2\).
Thus a unit normal to \(F_u\) is
\[
	\nu_u
	=
	\frac{\eta_u-P_u^{-1}\nabla u}
	{\sqrt{1+\left|P_u^{-1}\nabla u\right|^2}}.
\]
We choose the sign so that \(\nu_0=\nu\).
Let \(g_u\), \(h_u\), and \(d\mu_u\) be the induced metric, scalar second fundamental form, and volume measure, respectively, pulled back to \(M\), so that
\[
	\begin{aligned}
		g_u(X,Y)
		=
		\left\langle P_uX,P_uY\right\rangle
		+
		du(X)du(Y),\quad
		h_u(X,Y)
		=
		\left\langle
		\overline\nabla_{dF_u(X)}dF_u(Y),
		\nu_u
		\right\rangle.
	\end{aligned}
\]
Let \(A_u\) be the shape operator determined by \(g_u(A_uX,Y)=h_u(X,Y)\), or equivalently \(A_u=g_u^{-1}h_u\).
Set \(S_u=\tr(A_u^2)\) and \(f_{3,u}=\tr(A_u^3)\).
If \(\widehat H_u:F_u(M)\to\R\) denotes the mean curvature with respect to \(\nu_u\), define
\[
	\cH:\mathcal U\longrightarrow C^{k-2,\alpha}(M),
	\qquad
	\cH(u)=\widehat H_u\circ F_u=\operatorname{tr}_{g_u}h_u.
\]
Thus \(\cH(u)=0\) if and only if \(F_u\) is minimal, and the minimality of \(F\) gives \(\cH(0)=0\).
For the computation below, set \(v_u=\sqrt{1+\left|P_u^{-1}\nabla u\right|^2}\).
Let \(e_1,\ldots,e_n\) be a local \(g\)-orthonormal frame and write \(u_i=\nabla_{e_i}u\) and \(u_{ij}=\nabla^2u(e_i,e_j)\).
Let \(D\) denote the Euclidean connection of \(\R^{n+2}\).
Since \(\nu_u\perp F_u\), the second fundamental form may be computed with \(D\) in place of the spherical connection.
The Gauss formula for \(F\), together with the definitions of \(P_u,Q_u\), and \(\eta_u\), gives
\[
    \begin{aligned}
        D_X\eta_u&=-Q_uX-du(X)F_u,\\
        D_X(P_uY)
        &=P_u\nabla_XY-du(X)Q_uY-\sin u\,(\nabla_XA)Y\\
        &\quad+h(X,P_uY)\nu-\langle X,P_uY\rangle F.
    \end{aligned}
\]
Since \(h_u\) is tensorial, we may compute at a point where \(\nabla_{e_i}e_j=0\).
Using \(dF_u(e_i)=P_ue_i+u_i\eta_u\) and the formula for \(\nu_u\) in the definition of \(h_u\), together with the identities above, we obtain
\[
	(h_u)_{ij}
	=
	\frac{1}{v_u}u_{ij}
	+\frac{1}{v_u}\Big(
	\langle Q_ue_i,P_ue_j\rangle
	+u_i\langle Q_ue_j,\xi_u\rangle
	+u_j\langle Q_ue_i,\xi_u\rangle
	+\sin u\,\langle(\nabla_{e_i}A)e_j,\xi_u\rangle
	\Big).
\]
When \(u\) is constant, this formula reduces to the parallel hypersurface identity \(A_u=P_u^{-1}Q_u\), which also fixes the sign convention.
On the other hand, \((g_u)_{ij} = \left\langle P_ue_i,P_ue_j\right\rangle + u_i u_j\), so both \((g_u)_{ij}\) and its inverse \((g_u)^{ij}\) depend only on \(x,u,\nabla u\). Consequently
\[
    \begin{aligned}
        \cH(u)
        &=
        \frac{1}{v_u}(g_u)^{ij}u_{ij} \\
        &\quad +
        \frac1{v_u}(g_u)^{ij}
        \Big(
        \langle Q_ue_i,P_ue_j\rangle
        +u_i\langle Q_ue_j,\xi_u\rangle
        +u_j\langle Q_ue_i,\xi_u\rangle
        +\sin u\,
        \langle(\nabla_{e_i}A)e_j,\xi_u\rangle
    \Big).
    \end{aligned}
\]
Thus \(\cH\) is a quasilinear second-order operator.
The formulas above show that \(g_u\) and \(\frac{d\mu_u}{d\mu}\) depend only on \(u\) and \(\nabla u\), while \(h_u\), and hence \(A_u\), \(S_u\), and \(f_{3,u}\), also involve \(\nabla^2u\).
Thus \(g_u, \frac{d\mu_u}{d\mu}\in C^{k-1,\alpha}\) and \(h_u, A_u, S_u, f_{3,u}\in C^{k-2,\alpha}\).
On \(\mathcal U\), both \(P_u\) and \(g_u\) remain uniformly invertible.
Matrix inversion and the positive square root are analytic on their respective domains, while H\"older spaces are Banach algebras under pointwise multiplication.
The formulas above therefore show that all these quantities depend real analytically on \(u\), with values in the stated H\"older spaces.
For analytic composition operators in Schauder spaces, see \cite[Ch.~II]{Valent1988}.
In particular, \(\cH:\mathcal U\to C^{k-2,\alpha}(M)\) is real analytic near the origin.
Lemma~\ref{lem:variation} gives its derivative
\begin{equation}\label{eq:jacobi-operator}
	D\cH(0)=L:=\Delta+n+S.
\end{equation}

For smooth real-valued functions \(u,v\) on \(M\), integration by parts, together with the fact that \(M\) has no boundary and \(n+S\) is real-valued, gives
\[
	\int_M vLu d \mu
	=
	-\int_M\langle\nabla v,\nabla u\rangle d \mu
	+\int_M(n+S)uv d \mu
	=
	\int_M uLv d \mu.
\]
Thus \(L\) is formally self-adjoint. Since its principal part is \(\Delta\), it is also elliptic. Standard elliptic Schauder theory on the closed manifold \(M\) shows that \(L:C^{k,\alpha}(M)\to C^{k-2,\alpha}(M)\) is Fredholm: its kernel is finite-dimensional, its range is closed, and the range has finite codimension in the codomain. We write \(\operatorname{Ran}L\) for this range.

The Fredholm alternative, formal self-adjointness, and elliptic regularity imply that \(f\in C^{k-2,\alpha}(M)\) belongs to \(\operatorname{Ran}L\) if and only if
\(\int_M f\varphi d \mu=0\) for every \(\varphi\in\ker L\). It follows that
\(\operatorname{codim}\operatorname{Ran}L=\dim\ker L\). Hence the Fredholm index of \(L\), defined by
\(\operatorname{ind}L:=\dim\ker L-\operatorname{codim}\operatorname{Ran}L\), is zero.

Put \(K=\ker L\). Elliptic regularity shows that every element of the finite-dimensional space \(K\) is smooth. Let \(K^{\perp_{L^2}}\) denote its orthogonal complement with respect to the \(L^2(M,g)\) inner product, and set
\[
	X=C^{k,\alpha}(M)\cap K^{\perp_{L^2}},
	\qquad
	Y=C^{k-2,\alpha}(M)\cap K^{\perp_{L^2}}.
\]
Then \(C^{k,\alpha}(M)=K\oplus X\) and \(C^{k-2,\alpha}(M)=K\oplus Y\). Since \(L\) vanishes on \(K\), the preceding characterization of its range shows that \(L:X\to Y\) is bijective.
The Schauder estimate then gives a bounded inverse \(L^{-1}:Y\to X\).

Let \(P\) be the \(L^2\)-orthogonal projection onto \(K\), and put \(Q=I-P\).
Choosing an \(L^2\)-orthonormal basis of the smooth finite-dimensional space \(K\) shows that \(P\) is bounded on both \(C^{k,\alpha}(M)\) and \(C^{k-2,\alpha}(M)\).
The analytic implicit function theorem therefore supplies neighborhoods \(U_K\subset K\) and \(U_X\subset X\) of zero and a real analytic map
\[
	\Psi:U_K\longrightarrow U_X,
	\qquad
	\Psi(0)=0,
	\qquad
	D\Psi(0)=0,
\]
such that
\begin{equation}\label{eq:complement-equation}
	Q\cH(z+w)=0
	\quad\Longleftrightarrow\quad
	w=\Psi(z)
\end{equation}
whenever \(z+w\) is sufficiently small.
Define the finite-dimensional obstruction map
\begin{equation}\label{eq:kuranishi-map}
	\kappa:U_K\longrightarrow K,
	\qquad
	\kappa(z)=P\cH\bigl(z+\Psi(z)\bigr).
\end{equation}
It follows from \eqref{eq:complement-equation} and \eqref{eq:kuranishi-map} that the small minimal graphs are exactly those represented by \(z\in U_K\) with \(\kappa(z)=0\).

For \(u\in\mathcal U\), denote the \( d \mu_u\)-averages of \(S_u\) and \(f_{3,u}\) by
\begin{equation}\label{eq:averages}
	\overline S(u)
	=
	\frac{\int_M S_u d \mu_u}{\int_M d\mu_u},
	\qquad
	\overline f_3(u)
	=
	\frac{\int_M f_{3,u} d \mu_u}{\int_M d \mu_u}.
\end{equation}
Their deviations from these averages are measured by
\begin{align}
	\cV_S(u)
	&=
	\int_M\bigl(S_u-\overline S(u)\bigr)^2 d \mu_u,
	\label{eq:variance-S}\\
	\cV_3(u)
	&=
	\int_M\bigl(f_{3,u}-\overline f_3(u)\bigr)^2 d \mu_u.
	\label{eq:variance-f3}
\end{align}
The preceding formulas show that \(\overline S(u)\), \(\overline f_3(u)\), \(\cV_S(u)\), and \(\cV_3(u)\) depend real analytically on \(u\) near the origin. Here we use the Banach-algebra property of H\"older spaces and the fact that \(\int_M d \mu_u>0\).
Since the integrands in \eqref{eq:variance-S} and \eqref{eq:variance-f3} are continuous and nonnegative, \(\cV_S(u)=0\) exactly when \(S_u\) is spatially constant, and \(\cV_3(u)=0\) exactly when \(f_{3,u}\) is spatially constant.

\begin{theorem}\label{thm:local-rigidity}
    Let \(F:M\hookrightarrow\sphere^{n+1}(1)\) be a closed connected embedded minimal hypersurface with \(S\equiv s_0\) and constant \(f_3\).
    Then there exists \(\varepsilon>0\) such that, for every \(u\in C^{k,\alpha}(M)\) with \(\|u\|_{C^{k,\alpha}}<\varepsilon\), if \(F_u\) is minimal and both \(S_u\) and \(f_{3,u}\) are spatially constant, then \(S_u\equiv s_0\).
    Equivalently, the same conclusion holds in the local normal slice described above, and hence in a neighborhood of \([F]\) in \(\operatorname{Emb}^{k,\alpha}(M,\sphere^{n+1}(1)) \big/ \operatorname{Diff}^{k,\alpha}(M)\).
\end{theorem}

\begin{proof}
		Choose linear coordinates on the finite-dimensional space \(K\), and define
	\[
		u(z)=z+\Psi(z),
		\qquad
		s(z)=\overline S(u(z)),
	\]
	and consider the real analytic set
	\begin{equation*}
		Z=\left\{z\in U_K:
		\kappa(z)=0,
		\ \cV_S(u(z))=0,
		\ \cV_3(u(z))=0
		\right\}.
	\end{equation*}
    By construction, \(z\in Z\) if and only if \(F_{u(z)}\) is minimal and both \(S_{u(z)}\) and \(f_{3,u(z)}\) are spatially constant.
	Assume otherwise.
    Then there is a sequence \(z_j\in Z\) converging to \(0\) such that \(s(z_j)\neq s(0)=s_0\).
	After passing to a subsequence, either \(s(z_j)>s_0\) for every \(j\) or \(s(z_j)<s_0\) for every \(j\).
	We consider the first case; the second is handled in the same way.
	The origin then lies in the closure of the semianalytic set
	\[
		Z_+=Z\cap\{z:s(z)>s_0\}.
	\]
    The set \(Z_+\) is semianalytic, hence subanalytic, so Hironaka's curve selection lemma \cite[Proposition~3.9, p.~482]{Hironaka1973} gives \(\varepsilon>0\) and a real analytic curve \(\gamma:(-\varepsilon,\varepsilon)\to U_K\) such that \(\gamma(0)=0\) and \(\gamma((0,\varepsilon))\subset Z_+\).
    For \(0<t<\varepsilon\), the graphs \(F_{u(\gamma(t))}\) form a real analytic family of minimal embeddings with spatially constant \(S\) and \(f_3\).
    Since \(k\geq4\) was fixed above, this family is \(C^1\) with values in \(C^3(M,\sphere^{n+1}(1))\).
    Proposition~\ref{prop:path-rigidity} therefore applies on \((0,\varepsilon)\).
	Proposition~\ref{prop:path-rigidity} shows that \(s(\gamma(t))\) has zero derivative for every \(0<t<\varepsilon\), and hence is constant on that interval.
    Since \(\gamma(t)\to0\) as \(t\to0\), continuity of \(s\) gives \(s(\gamma(t))=s_0\) for \(0<t<\varepsilon\). This contradicts \(\gamma(t)\in Z_+\), where \(s(\gamma(t))>s_0\).
    Hence every sufficiently small normal graph that is minimal and has spatially constant \(S\) and \(f_3\) satisfies \(S\equiv s_0\).
    The local normal-graph representation then gives the same conclusion for embeddings sufficiently close to \(F\), modulo reparametrization.
\end{proof}

\begin{remark}\label{rem:singular-moduli}
    The Lyapunov--Schmidt reduction allows for a nontrivial Jacobi kernel \(K=\ker L\).
    Curve selection handles possible singularities of the finite-dimensional constrained analytic set \(Z\).
	Nontriviality of \(K\) alone does not imply that the corresponding point of the minimal moduli space is singular.
    Vanishing of the derivative along smooth paths alone does not exclude the accumulation of different \(S\)-values along singular branches.
\end{remark}

\section{Uniform tubes and smooth embedded compactness}\label{sec:compactness}

We now prove the global compactness statement needed to allow the source manifold and its topology to vary.
For a closed embedded hypersurface \(F:M\hookrightarrow\sphere^{n+1}(1)\), define its spherical normal injectivity radius to be the supremum of the numbers \(r>0\) for which
\[
	E:M\times(-r,r)\longrightarrow\sphere^{n+1}(1),
	\qquad
	E(x,t)=\cos t\,F(x)+\sin t\,\nu(x),
\]
is injective and nonsingular.

\begin{lemma}\label{lem:uniform-tube}
	Let \(F:M^n\hookrightarrow\sphere^{n+1}(1)\) be a closed connected embedded minimal hypersurface satisfying \(|A|\leq K\).
	Define
	\begin{equation}\label{eq:rK}
		r_K=
		\begin{cases}
			\arctan(K^{-1}),&K>0,\\
			\pi/2,&K=0.
		\end{cases}
	\end{equation}
	Then the spherical normal injectivity radius of \(F(M)\) is at least \(r_K\).
\end{lemma}

\begin{proof}
    The embedded connected hypersurface separates the sphere into two connected open domains \(\Omega_+\) and \(\Omega_-\), whose closures are compact manifolds with common boundary \(F(M)\).
    Give each closure the metric induced from the sphere and choose the inward unit normal along its boundary.
    With the induced metric, \(\overline{\Omega_\pm}\) has Ricci tensor \(ng>0\), while its boundary \(F(M)\) has zero mean curvature.
    Fix either sign.
    For the compact manifold \(\overline{\Omega_\pm}\), let \(\operatorname{Roll}\) be the infimum of the inward boundary cut distances and let \(\operatorname{Foc}\) be the infimum of the first inward focal distances.
    By these definitions, the inward normal exponential map is a diffeomorphism from \(\partial\Omega_\pm\times[0,r)\) onto its image whenever \(r<\operatorname{Roll}\), and one always has \(\operatorname{Roll}\leq\operatorname{Foc}\).
    Howard's rigidity form of the rolling theorem states that, for a complete connected manifold with smooth non-empty compact boundary, nonnegative Ricci curvature, and nonnegative inward mean curvature, the strict inequality \(\operatorname{Roll}<\operatorname{Foc}\) can occur only for a Riemannian cylinder or a generalized M\"obius band \cite[Theorem~3]{Howard1999}.
    Each exceptional space has a local product interval direction with zero Ricci curvature, so neither can be isometric to \(\overline{\Omega_+}\) or \(\overline{\Omega_-}\).
    Consequently the rolling radius of each side equals its focal radius.

	Along a normal geodesic in the unit sphere, the tangential normal Jacobi map is
	\begin{equation}\label{eq:jacobi-map-tube}
		J_t=\cos t\,I-\sin t\,A
	\end{equation}
	for the corresponding inward shape operator.
		If \(\lambda\) is a principal curvature and \(K>0\), the first positive zero of \(\cos t-\lambda\sin t\), when it occurs before \(\pi/2\), is \(\arctan(1/\lambda)\) with \(0<\lambda\leq K\). If \(\lambda\leq0\), no positive zero occurs before \(\pi/2\).
	Thus the first focal distance on either side is at least \(\arctan(K^{-1})\).
	When \(K=0\), Equation~\eqref{eq:jacobi-map-tube} is \(\cos t\,I\), and the first focal distance is \(\pi/2\).
	The equality of rolling and focal radii shows that each one-sided normal exponential map is injective and nonsingular before \(r_K\).
    For these parameter values, the normal segments are the unique minimizing geodesics to the boundary and remain in the interior of the corresponding domain.
    Hence normal geodesics directed into opposite sides cannot meet, because the two one-sided images lie in disjoint domains.
	The two one-sided estimates therefore give the required two-sided normal injectivity bound.
\end{proof}

\begin{proposition}\label{prop:area-bound}
	Under the hypotheses of Lemma~\ref{lem:uniform-tube}, let \(r_K\) be defined by \eqref{eq:rK}, and put \(\delta_K=r_K/2\) and \(b_K=\cos\delta_K-K\sin\delta_K\).
	Then \(b_K>0\) and
	\begin{equation}\label{eq:area-bound}
		\Vol(M)
		\leq
		\frac{\Vol(\sphere^{n+1}(1))}{2\delta_K b_K^n}.
	\end{equation}
\end{proposition}

\begin{proof}
	The normal Jacobian at \((x,t)\) is
	\begin{equation}\label{eq:tube-jacobian}
		\mathcal J(x,t)=\det\bigl(\cos t\,I-\sin t\,A_x\bigr)
		=\prod_{i=1}^n\bigl(\cos t-\lambda_i(x)\sin t\bigr).
	\end{equation}
	For \(|t|\leq\delta_K\) and \(|\lambda_i|\leq K\), every factor in \eqref{eq:tube-jacobian} is at least \(b_K\).
	The choice \(\delta_K<r_K\) implies \(b_K>0\).
    Lemma~\ref{lem:uniform-tube} shows that the normal exponential map is injective on \(M\times[-\delta_K,\delta_K]\).
	The change-of-variables formula and \eqref{eq:tube-jacobian} therefore give
	\[
		\Vol(\sphere^{n+1}(1))
		\geq
		\int_{-\delta_K}^{\delta_K}\int_M\mathcal J(x,t) d \mu(x) d  t
		\geq
		2\delta_K b_K^n\Vol(M),
	\]
	which is \eqref{eq:area-bound}.
\end{proof}

For \(x\in\sphere^{n+1}(1)\) and \(\rho>0\), let \(B_\rho(x)\) denote the open geodesic ball of radius \(\rho\) centered at \(x\).
If \(P\subset T_x\sphere^{n+1}(1)\) is a linear subspace, set
\[
	B_\rho^P(0)=\{v\in P:|v|<\rho\},
\]
where the norm is induced by the spherical metric at \(x\).

\begin{lemma}\label{lem:uniform-local-graph}
	For fixed \(n\geq2\), \(K\geq0\), and \(0<\theta<1\), there are constants \(r>0\) and \(C<\infty\).
	If \(\Sigma^n\subset\sphere^{n+1}(1)\) is a closed connected embedded minimal hypersurface satisfying \(|A|\leq K\), then, for every \(x\in\Sigma\), the entire set \(\Sigma\cap B_r(x)\) is a single spherical graph over \(T_x\Sigma\).
	More precisely, after choosing a unit normal \(\nu(x)\), there is a smooth function
	\[
		u_x:B_{2r}^{T_x\Sigma}(0)\longrightarrow(-r,r)
	\]
	such that \(u_x(0)=0\), \(Du_x(0)=0\),
	\[
		\|u_x\|_{C^{1,1}(B_{2r}^{T_x\Sigma}(0))}\leq C,
		\qquad
		\|Du_x\|_{C^0(B_{2r}^{T_x\Sigma}(0))}\leq\theta,
	\]
	and
	\begin{equation}\label{eq:entire-local-graph}
		\Sigma\cap B_r(x)
		=
		\left\{
		\exp_x^\sphere\bigl(v+u_x(v)\nu(x)\bigr):
		v\in B_{2r}^{T_x\Sigma}(0)
		\right\}
		\cap B_r(x).
	\end{equation}
\end{lemma}

\begin{proof}
	Let \(r_K\) be defined by \eqref{eq:rK}, put \(\rho=\frac{r_K}{2}\), \(b_K=\cos\rho-K\sin\rho>0\), and choose a unit normal field \(\nu\) along \(\Sigma\).
	By Lemma~\ref{lem:uniform-tube}, the normal exponential map
	\[
		\mathcal E:\Sigma\times(-\rho,\rho)\longrightarrow\sphere^{n+1}(1),
		\qquad
		\mathcal E(x,t)=\cos t\,x+\sin t\,\nu(x),
	\]
	is a diffeomorphism onto its image \(U\).
	The image \(U\) is exactly the open metric \(\rho\)-neighborhood of \(\Sigma\).
	Indeed, a minimizing geodesic from a point at distance less than \(\rho\) to \(\Sigma\) is normal at its endpoint.
	Conversely, if \(\mathcal E(x,t)\) had distance less than \(|t|\) from \(\Sigma\), a minimizing normal geodesic would give a second representation under \(\mathcal E\), contradicting injectivity.
	Thus \(s\bigl(\mathcal E(x,t)\bigr)=t\) is the smooth signed-distance function on \(U\).

	In a principal frame at \(x\), the tangential eigenvalues of \(\nabla^2s\) at \(\mathcal E(x,t)\) are
	\[
		-\frac{\sin t+\lambda_i(x)\cos t}
		{\cos t-\lambda_i(x)\sin t},
		\qquad 1\leq i\leq n,
	\]
	and the eigenvalue in the normal geodesic direction is zero.
	Since
	\[
		\cos t-\lambda_i(x)\sin t
		\geq \cos\rho-K\sin\rho=b_K>0
	\]
	for \(|t|<\rho\), we have the uniform estimates
	\begin{equation}\label{eq:signed-distance-C11}
		|\nabla s|=1,
		\qquad
		|\nabla^2s|
		\leq
		C_0:=\frac{\sqrt n(1+K)}{b_K}
		\quad\text{on }U.
	\end{equation}

	Choose \(r>0\), depending only on \(n\) and \(K\), so small that \(4r<\rho\), \(16C_0r<1\), and the spherical exponential charts on balls of radius \(4r\) have uniformly controlled \(C^2\) norms and distortion.
	For \(x\in\Sigma\), define
	\[
		\Phi_x(v,\tau)
		=
		\exp_x^\sphere\bigl(v+\tau\nu(x)\bigr),
		\qquad
		v\in B_{2r}^{T_x\Sigma}(0),\quad |\tau|\leq r.
	\]
	The coordinate cylinder lies in \(B_{4r}(x)\subset U\).
	At \((0,0)\), one has \(\left.\partial_\tau(s\circ\Phi_x)\right|_{(0,0)}=1\).
	The Hessian bound \eqref{eq:signed-distance-C11} and the uniform \(C^2\) control of \(\Phi_x\) imply a uniform Lipschitz bound for \(d(s\circ\Phi_x)\).
	After decreasing \(r\), uniformly in \(x\) and \(\Sigma\), it follows that
	\begin{equation}\label{eq:vertical-monotonicity}
		\partial_\tau(s\circ\Phi_x)(v,\tau)\geq\frac14
	\end{equation}
	throughout the cylinder.
	Moreover, since \(s(x)=0\) and \(ds_x(v)=0\) for \(v\in T_x\Sigma\), Taylor's formula along the radial geodesic and \eqref{eq:signed-distance-C11} give \(\left|s\bigl(\Phi_x(v,0)\bigr)\right| \leq\frac12C_0|v|^2<\frac r8\).
	Integrating \eqref{eq:vertical-monotonicity} in the \(\tau\)-direction yields
	\[
		s\bigl(\Phi_x(v,-r)\bigr)<0
		<
		s\bigl(\Phi_x(v,r)\bigr).
	\]
	Every vertical coordinate line therefore contains exactly one zero of \(s\circ\Phi_x\).

	The quantitative implicit-function theorem produces a smooth function \(u_x\) whose graph is precisely the zero set in the coordinate cylinder.
    Since \(s(\Phi_x(0,0))=s(x)=0\), uniqueness of the zero on the vertical line gives \(u_x(0)=0\). Differentiating \((s\circ\Phi_x)\bigl(v,u_x(v)\bigr)=0\) at \(v=0\), and using \(ds_x|_{T_x\Sigma}=0\) and \(\partial_\tau(s\circ\Phi_x)(0,0)=1\), gives \(Du_x(0)=0\).
	The functions \(s\circ\Phi_x\) have a uniform \(C^{1,1}\) bound and their \(\tau\)-derivatives have the uniform positive lower bound \eqref{eq:vertical-monotonicity}.
	Differentiating \((s\circ\Phi_x)\bigl(v,u_x(v)\bigr)=0\) once and twice, the second time almost everywhere, gives a uniform \(C^{1,1}\) bound for \(u_x\).
	Let \(C_1\) be a uniform Lipschitz bound for \(Du_x\) furnished by the preceding \(C^{1,1}\) estimate.
    Since \(Du_x(0)=0\), for every \(v\in B_{2r}^{T_x\Sigma}(0)\), we have \(|Du_x(v)| \leq C_1|v| \leq 2C_1r\).
    After decreasing \(r\) further, depending also on \(\theta\), we may assume that \(2C_1r\leq\theta\).
    Hence
    \[
        \|Du_x\|_{C^0(B_{2r}^{T_x\Sigma}(0))}
        \leq
        \theta.
    \]

	If \(y\in\Sigma\cap B_r(x)\), its unique spherical exponential coordinate at \(x\) has the form \(v+\tau\nu(x)\), with \(|v|,|\tau|<r\).
	Since \(s(y)=0\), uniqueness of the zero on the corresponding vertical line gives \(\tau=u_x(v)\).
	The converse inclusion is immediate because the graph is contained in \(s^{-1}(0)=\Sigma\).
	This proves \eqref{eq:entire-local-graph}, as we hoped.
\end{proof}

\begin{proposition}\label{prop:smooth-compactness}
	Let \(F_j:M_j^n\hookrightarrow\sphere^{n+1}(1)\) be a sequence of closed connected embedded minimal hypersurfaces satisfying
	\[
		\sup_{M_j}|A_j|\leq K.
	\]
	After passage to a subsequence, there are a closed connected smooth manifold \(M_\infty\), a smooth minimal embedding \(F_\infty:M_\infty\hookrightarrow\sphere^{n+1}(1)\), and diffeomorphisms \(\psi_j:M_\infty\to M_j\) such that \(F_j\circ\psi_j\to F_\infty\) smoothly.
    Moreover, the images \(F_j(M_j)\) are single-valued spherical normal graphs over \(F_\infty(M_\infty)\) for all sufficiently large \(j\).
\end{proposition}

\begin{proof}
	Set \(\Sigma_j=F_j(M_j)\).
    Apply Lemma~\ref{lem:uniform-local-graph} with \(\theta=1/10\), and let \(r\) be the resulting radius.
	The Blaschke selection theorem \cite[Theorem~7.3.8, p.~253]{BuragoBuragoIvanov2001} gives, after passage to a subsequence, Hausdorff convergence
	\[
		\Sigma_j\longrightarrow\Sigma_\infty
	\]
	for a nonempty compact set \(\Sigma_\infty\subset\sphere^{n+1}(1)\).
	Choose points \(x_\infty^1,\ldots,x_\infty^N\in\Sigma_\infty\) such that the balls \(B_{r/8}(x_\infty^a)\) cover \(\Sigma_\infty\), and choose \(x_j^a\in\Sigma_j\) with \(x_j^a\to x_\infty^a\).
    Using parallel transport, identify the tangent spaces with \(T_{x_\infty^a}\sphere^{n+1}(1)\).
    After passing to a further subsequence, we may assume, for all \(1\leq a\leq N\), that
    \[
        T_{x_j^a}\Sigma_j\longrightarrow P_\infty^a
        \subset T_{x_\infty^a}\sphere^{n+1}(1).
    \]

	The local graphs supplied by Lemma~\ref{lem:uniform-local-graph} have uniformly small gradients and uniform \(C^{1,1}\) bounds.
    Using parallel transport, identify the domains of the local graphs with fixed balls in \(P_\infty^a\), and regraph over \(P_\infty^a\) when necessary.
    The Arzelà--Ascoli theorem then gives, for every \(\alpha\in(0,1)\),
    \[
        u_j^a\longrightarrow u_\infty^a
        \qquad\text{in }C^{1,\alpha},
    \]
    where \(u_\infty^a\) is \(C^{1,1}\) with the same uniform bound.
	The limit graph is exactly \(\Sigma_\infty\) in \(B_{r/2}(x_\infty^a)\).
    For one inclusion, every point on the limit graph is a limit of points of \(\Sigma_j\).
    For the reverse inclusion, let \(y\in\Sigma_\infty\cap B_{r/2}(x_\infty^a)\) and choose \(y_j\in\Sigma_j\) with \(y_j\to y\).
    For large \(j\), \eqref{eq:entire-local-graph} shows that \(y_j\) lies on the graph centered at \(x_j^a\).
    Hence there is no additional component or sheet in the smaller ball.

    For large \(j\), the balls \(B_{r/4}(x_j^a)\) cover \(\Sigma_j\), by Hausdorff convergence and the choice of the points \(x_\infty^a\).
    On overlaps, the limiting graphs agree, since they describe the same Hausdorff limit.
    They therefore define an embedded \(C^{1,1}\) atlas on \(\Sigma_\infty\), and the convergence \(\Sigma_j\to\Sigma_\infty\) is locally one-sheeted in \(C^{1,\alpha}\).
    Since a Hausdorff limit of connected compact sets is connected, \(\Sigma_\infty\) is connected.

    Proposition~\ref{prop:area-bound} gives a uniform mass bound. The multiplicity-one conclusion, however, comes from the finite one-sheet graph cover rather than from the mass bound alone. 
    Regard each associated varifold as a Radon measure on the compact Grassmann bundle \(G_n(T\sphere^{n+1})\).
    On each one-sheet graph chart, the graph parametrizations, area densities, and tangent planes converge uniformly.
    Using a partition of unity subordinate to the smaller balls, we may test this convergence against any continuous function on the Grassmann bundle.
    It follows that \(\Sigma_j\to\Sigma_\infty\) as varifolds with multiplicity one.
    For every smooth vector field \(X\) on the sphere, the function \((x,P)\mapsto\operatorname{div}_P X(x)\) is continuous on the Grassmann bundle.
    Minimality and the preceding varifold convergence therefore imply
	\[
		0
		=
		\lim_{j\to\infty}
		\int_{\Sigma_j}\operatorname{div}_{T\Sigma_j}X\, d \mu_j
		=
		\int_{\Sigma_\infty}
		\operatorname{div}_{T\Sigma_\infty}X\, d \mu_\infty.
	\]
	Hence \(\Sigma_\infty\) is stationary.

    In a \(C^{1,1}\) graph chart, stationarity implies that the graph function \(u\) is a weak solution of the Euler--Lagrange equation for the area functional
    \begin{equation}\label{eq:weak-minimal-graph}
        \partial_\beta\mathcal A^\beta(x,u,\nabla u)
        =
        \mathcal B(x,u,\nabla u),
    \end{equation}
    where the equation is understood in the sense of distributions.
    Here \(\mathcal A^\beta\) and \(\mathcal B\) are smooth.
    Under the uniform gradient bound, the matrix \(\left( \frac{\partial\mathcal A^\beta}{\partial p_\gamma} \right)_{\beta,\gamma=1}^n\) is uniformly positive definite.
    Since \(u\in C^{1,1}\subset W^{2,\infty}\), the Sobolev chain rule expands \eqref{eq:weak-minimal-graph} almost everywhere into the nondivergence equation
    \[
        a^{\beta\gamma}(x)u_{\beta\gamma}
        =
        f(x),
    \]
    where
    \[
        a^{\beta\gamma}(x)
        =
        \frac{\partial\mathcal A^\beta}{\partial p_\gamma}
        (x,u(x),\nabla u(x))
    \]
    and
    \[
        f(x)
        =
        \mathcal B(x,u,\nabla u)
        -\frac{\partial\mathcal A^\beta}{\partial x^\beta}(x,u,\nabla u)
        -\frac{\partial\mathcal A^\beta}{\partial z}(x,u,\nabla u)u_\beta.
    \]
    Thus \(a^{\beta\gamma}\) and \(f\) are Lipschitz, and the equation is uniformly elliptic.
    In particular, on every relatively compact subchart and for each \(\alpha\in(0,1)\), they belong to \(C^{0,\alpha}\). Since \(u\in W^{2,\infty}\) is a strong solution of the equation almost everywhere, the interior strong-solution regularity theorem for uniformly elliptic nondivergence equations gives \(u\in C^{2,\alpha}\) \cite[Theorem~9.19]{GilbargTrudinger2001}.
	Once this regularity is available, differentiating the equation and applying Schauder estimates iteratively gives \(u\in C^\infty\).
	Therefore \(\Sigma_\infty\) is a closed connected smooth embedded minimal hypersurface.

	Let \(\pi_\infty\) be the nearest-point projection from a fixed tubular neighborhood of the smooth compact hypersurface \(\Sigma_\infty\) onto \(\Sigma_\infty\).
	Hausdorff and local \(C^1\) convergence imply that \(\Sigma_j\) lies in this tube and is transverse to its normal fibers for all sufficiently large \(j\).
	Thus
	\[
		\pi_\infty|_{\Sigma_j}:\Sigma_j\longrightarrow\Sigma_\infty
	\]
	is a local diffeomorphism.
	Its image is open and closed in the connected hypersurface \(\Sigma_\infty\), so the map is surjective.
    Since \(\Sigma_j\) is compact, it is therefore a finite covering.

	We now verify that this covering has degree one.
	Fix \(y\in\Sigma_\infty\) and choose one of the smaller one-sheet convergence balls containing \(y\).
	After shrinking that ball once, both \(\Sigma_\infty\) and the entire portion of \(\Sigma_j\) in the twice larger ball are graphs over a fixed convex ball in \(T_y\Sigma_\infty\), and their graph functions converge in \(C^1\).
	In these coordinates, \(\pi_\infty|_{\Sigma_j}\) is represented by a map \(\Pi_j\) satisfying
    \[
        \|D\Pi_j-I\|_{C^0}<\frac12
    \]
    for all sufficiently large \(j\).
    For points \(v,w\) in the smaller convex coordinate ball, integration along the segment from \(w\) to \(v\) gives
    \[
        |\Pi_j(v)-\Pi_j(w)|
        \geq\frac12|v-w|.
    \]
    Hence \(\Pi_j\) is injective there.
    If \(z\in\Sigma_j\) lies in the fiber over \(y\), then \(y=\pi_\infty(z)\) and hence
    \[
        d(z,y)=d(z,\Sigma_\infty)
        \leq d_H(\Sigma_j,\Sigma_\infty)\longrightarrow0.
    \]
    Thus every point of the fiber lies in this coordinate ball when \(j\) is large.
	The fiber therefore contains at most one point, and surjectivity shows that it contains exactly one.
	Since \(y\) was arbitrary, the covering has degree one and \(\pi_\infty|_{\Sigma_j}\) is a diffeomorphism.

	Set \(M_\infty=\Sigma_\infty\), let \(F_\infty\) be the inclusion, and define
	\[
		G_j=\bigl(\pi_\infty|_{\Sigma_j}\bigr)^{-1},
		\qquad
		\psi_j=F_j^{-1}\circ G_j.
	\]
	Then \(\psi_j:M_\infty\to M_j\) is a diffeomorphism and \(F_j\circ\psi_j=G_j\).
	After choosing a unit normal \(\nu_\infty\), there is a unique smooth function \(v_j:M_\infty\to\R\), with \(\|v_j\|_{C^0}\to0\), such that
	\begin{equation}\label{eq:global-normal-heights}
		F_j\circ\psi_j(x)
		=
		\cos v_j(x)\,F_\infty(x)
		+
		\sin v_j(x)\,\nu_\infty(x).
	\end{equation}
	The local \(C^{1,\alpha}\) convergence gives \(v_j\to0\) in \(C^{1,\alpha}\).
	On the finite atlas above, the passage from the tangent-plane graphs in Lemma~\ref{lem:uniform-local-graph} to the normal graph \eqref{eq:global-normal-heights} is given by a fixed smooth coordinate change.
	Uniform transversality bounds its inverse derivatives, so the uniform local \(C^{1,1}\) estimates transfer to a uniform \(C^{1,1}\) bound for \(v_j\).

	In a fixed finite atlas of \(M_\infty\), the minimal graph equation takes the form
    \[
        a_\infty^{\beta\gamma}(x,v_j,\nabla v_j)
        \nabla_\beta\nabla_\gamma v_j
        =
        f_\infty(x,v_j,\nabla v_j),
    \]
    where the coefficients \(a_\infty^{\beta\gamma}\) and \(f_\infty\) depend smoothly on \(x\), \(v_j\), and \(\nabla v_j\).
    The equation is uniformly elliptic.
    The uniform \(C^{1,1}\) bound makes the coefficients and the right-hand side uniformly bounded in \(C^{0,\alpha}\).
    Applying the interior strong-solution estimate on a finite collection of slightly smaller charts yields a uniform \(C^{2,\alpha}\) bound.
    Schauder bootstrapping then yields uniform \(C^{m,\alpha}\) bounds for every \(m\).
    Thus the sequence \(\{v_j\}\) is precompact in \(C^\infty(M_\infty)\).
    Since \(v_j\to0\) in \(C^1\), every smooth subsequential limit is zero, and hence \(v_j\to0\) smoothly.
    By \eqref{eq:global-normal-heights}, \(F_j\circ\psi_j\to F_\infty\) smoothly, and the images are single-valued spherical normal graphs over \(F_\infty(M_\infty)\) for all sufficiently large \(j\).
\end{proof}

\begin{remark}\label{rem:compactness-logic}
	A uniform bound for \(|A|\) gives local graphical control.
    The global argument also uses minimality and embeddedness through Howard's rolling theorem.
    This gives a uniform tubular neighborhood, which in turn yields the area bound and rules out multiple sheets in the limit.

	Both assumptions are needed.
    Wiygul constructed embedded minimal surfaces by stacking Clifford tori that converge to the Clifford torus with any prescribed fixed multiplicity \cite{Wiygul2020}.
    In these examples the catenoidal necks shrink, so no uniform bound for \(|A|\) can hold.
    On the other hand, if immersions are allowed, one may precompose a standard Clifford embedding with an \(\sphere^1\)-factor by a degree-\(d\) covering of that factor.
    The resulting immersions have the same \(|A|\), \(S\), and \(f_3\), while their area and covering multiplicity increase with \(d\).
    Thus the curvature bound excludes the first behavior, while embeddedness excludes the second.
\end{remark}

\begin{corollary}
	For fixed \(n\geq2\) and \(\Lambda\geq0\), only finitely many diffeomorphism types occur among closed connected manifolds \(M^n\) admitting a minimal embedding \(F:M\hookrightarrow\sphere^{n+1}(1)\) with \(\sup_M|A|^2\leq\Lambda\).
\end{corollary}

\begin{proof}
    Suppose otherwise, choose minimal embeddings \(F_j:M_j\hookrightarrow\sphere^{n+1}(1)\) such that no two of the manifolds \(M_j\) are diffeomorphic and \(\sup_{M_j}|A_j|\leq\sqrt{\Lambda}\).
	Proposition~\ref{prop:smooth-compactness} gives, after passing to a subsequence, a smooth manifold \(M_\infty\) and diffeomorphisms \(\psi_j:M_\infty\to M_j\).
	This contradicts the choice of the \(M_j\).
\end{proof}

\begin{remark}
	The uniformity of the curvature bound in the preceding corollary is essential.
	Lawson constructed closed embedded minimal surfaces \(\Sigma_g\subset\sphere^3\) of arbitrarily large genus \(g\) \cite{Lawson1970}.
	Writing \(S_g=|A_g|^2\), the preceding corollary implies that
	\[
		\max_{\Sigma_g}S_g\longrightarrow\infty
		\qquad\text{as }g\longrightarrow\infty.
	\]
	Indeed, otherwise a subsequence would have a uniform upper bound for \(S_g\) while representing infinitely many diffeomorphism types.
	This observation concerns the larger compactness class and does not assert that the Lawson surfaces have spatially constant \(S\) or \(f_3\).
\end{remark}

\section{Proof of the global theorem}\label{sec:global-proof}

\begin{proof}[Proof of Theorem~\ref{thm:main}]
	Assume for contradiction that \(\mathcal E_n^{(3)}\cap[0,\Lambda]\) contains infinitely many distinct values \(s_j\).
	For every \(j\), choose a closed connected manifold \(M_j\) and a minimal embedding \(F_j:M_j\hookrightarrow\sphere^{n+1}(1)\) such that \(S_j\equiv s_j\) and \(f_{3,j}\) is spatially constant.
	The bound \(s_j\leq\Lambda\) gives \(|A_j|\leq\sqrt\Lambda\).
	Using Proposition~\ref{prop:smooth-compactness}, we obtain a smoothly convergent subsequence with an embedded minimal limit \(F_\infty:M_\infty\hookrightarrow\sphere^{n+1}(1)\), and the members of the subsequence are eventually single normal graphs over the limit.

	After passing to a further subsequence, compactness of \([0,\Lambda]\) gives \(s_j\to s_\infty\).
	Smooth convergence of the second fundamental forms yields \(S_\infty = s_\infty\).
	Choose the normals of the eventual graphs compatibly with a fixed normal of the limit.
	Each corresponding \(f_{3,j}\) is still spatially constant, and smooth convergence gives a spatially constant limit \(f_{3,\infty}\).
	Thus the limiting hypersurface satisfies all hypotheses of Theorem~\ref{thm:local-rigidity}.
	That theorem implies that for every sufficiently large \(j\), \(F_j(M_j)\) has constant squared norm equal to \(s_\infty\).
	Hence \(s_j=s_\infty\) for all sufficiently large \(j\), contradicting the choice of pairwise distinct values.
	The contradiction proves finiteness on \([0,\Lambda]\).
\end{proof}

\begin{corollary}\label{cor:discrete}
	For each fixed \(n\geq2\), the set $\mathcal{E}_n^{(3)}$ is a closed locally finite subset of \([0,\infty)\), and in particular it has no finite accumulation point.
\end{corollary}

\begin{proof}
	Local finiteness is exactly the conclusion of Theorem~\ref{thm:main} on arbitrary compact intervals.
	Every locally finite subset of \([0,\infty)\) is closed because a point in its closure lies in a compact interval containing only finitely many points of the subset.
\end{proof}

\begin{remark}\label{rem:scope}
	The set $\mathcal{E}_n^{(3)}$ is a union over all closed connected source manifolds, not a spectrum attached to one fixed \(M\).
	If disconnected sources are allowed while \(S\) and \(f_3\) are required to have one global constant value, selecting any connected component shows that the value set is unchanged.
	Theorem~\ref{thm:main} is therefore the weak Chern conclusion for the subclass defined by embeddedness and constant \(f_3\), but it does not settle the original immersed problem.
\end{remark}

\noindent
\textbf{AI disclosure.}
During the preparatory phase, we used ChatGPT 5.6 Sol to interpret relevant literature and mathematical tools, including the Blaschke selection theorem and the real analytic curve selection lemma. All mathematical arguments and proofs were independently developed and verified by the authors, who prepared the final manuscript and take full responsibility for its content.


\begin{thebibliography}{99}

	\bibitem{AlmeidaBrito1990} S.~C.~de Almeida and F.~G.~B.~Brito, Closed 3-dimensional hypersurfaces with constant mean curvature and constant scalar curvature, \emph{Duke Math. J.} \textbf{61} (1990), no.~1, 195--206.

    \bibitem{BuragoBuragoIvanov2001} D.~Burago, Y.~Burago, and S.~Ivanov, \emph{A Course in Metric Geometry}, Graduate Studies in Mathematics, vol.~33, American Mathematical Society, Providence, RI, 2001.

	\bibitem{Cartan1939} \'{E}.~Cartan, Sur des familles remarquables d'hypersurfaces isoparam\'etriques dans les espaces sph\'eriques, \emph{Math. Z.} \textbf{45} (1939), 335--367.

	\bibitem{Chang1993} S.~P.~Chang, On minimal hypersurfaces with constant scalar curvatures in \(\sphere^4\), \emph{J. Differential Geom.} \textbf{37} (1993), 523--534.

	\bibitem{ChengWeiYamashiro2025} Q.-M.~Cheng, G.~X.~Wei, and T.~Yamashiro, The second gap of the scalar curvature of complete minimal hypersurfaces, \emph{Comm. Anal. Geom.} \textbf{33} (2025), no.~3, 623--636.

	\bibitem{Chern1968} S.~S.~Chern, \emph{Minimal Submanifolds in a Riemannian Manifold}, Technical Report No.~19, Department of Mathematics, University of Kansas, Lawrence, 1968.

	\bibitem{ChernDoCarmoKobayashi1970} S.~S.~Chern, M.~do Carmo, and S.~Kobayashi, Minimal submanifolds of a sphere with second fundamental form of constant length, in \emph{Functional Analysis and Related Fields}, Springer, New York, 1970, pp.~59--75.

	\bibitem{DengGuWei2017} Q.~T.~Deng, H.~L.~Gu, and Q.~Y.~Wei, Closed Willmore minimal hypersurfaces with constant scalar curvature in \(\sphere^5(1)\) are isoparametric, \emph{Adv. Math.} \textbf{314} (2017), 278--305.

	\bibitem{DingXin2011} Q.~Ding and Y.~L.~Xin, On Chern's problem for rigidity of minimal hypersurfaces in the spheres, \emph{Adv. Math.} \textbf{227} (2011), 131--145.

    \bibitem{FiresterTsiamis2026}
    B.~Firester and R.~Tsiamis,
    On Chern's conjecture for minimal submanifolds of the sphere, arXiv:2608.18074 (2026).

	\bibitem{GeTang2012} J.~Q.~Ge and Z.~Z.~Tang, Chern conjecture and isoparametric hypersurfaces, in \emph{Differential Geometry}, Adv. Lect. Math. (ALM), vol.~22, International Press, Somerville, MA, 2012, pp.~49--60.

	\bibitem{GilbargTrudinger2001} D.~Gilbarg and N.~S.~Trudinger, \emph{Elliptic Partial Differential Equations of Second Order}, 2nd ed., Classics in Mathematics, Springer-Verlag, Berlin, 2001.

    \bibitem{Hironaka1973} H.~Hironaka, Subanalytic sets, in \emph{Number Theory, Algebraic Geometry and Commutative Algebra: In Honor of Yasuo Akizuki}, Kinokuniya, Tokyo, 1973, pp.~453--493.

	\bibitem{Howard1999} R.~Howard, Blaschke's rolling theorem for manifolds with boundary, \emph{Manuscripta Math.} \textbf{99} (1999), no.~4, 471--483.

	\bibitem{Lawson1969} H.~B.~Lawson, Jr., Local rigidity theorems for minimal hypersurfaces, \emph{Ann. of Math. (2)} \textbf{89} (1969), 187--197.

    \bibitem{Lawson1970} H.~B.~Lawson, Jr., Complete minimal surfaces in \(\sphere^3\), \emph{Ann. of Math. (2)} \textbf{92} (1970), no.~3, 335--374.

	\bibitem{LeiXuXu2021} L.~Lei, H.~W.~Xu, and Z.~Y.~Xu, On the generalized Chern conjecture for hypersurfaces with constant mean curvature in a sphere, \emph{Sci. China Math.} \textbf{64} (2021), no.~7, 1493--1504.

	\bibitem{Munzner1980} H.~F.~M\"unzner, Isoparametrische Hyperfl\"achen in Sph\"aren, \emph{Math. Ann.} \textbf{251} (1980), 57--71.

	\bibitem{Munzner1981} H.~F.~M\"unzner, Isoparametrische Hyperfl\"achen in Sph\"aren. II, \emph{Math. Ann.} \textbf{256} (1981), 215--232.

	\bibitem{PengTerng1983a} C.~K.~Peng and C.~L.~Terng, Minimal hypersurfaces of spheres with constant scalar curvature, in \emph{Seminar on Minimal Submanifolds}, Ann. of Math. Stud., vol.~103, Princeton University Press, Princeton, NJ, 1983, pp.~177--198.

	\bibitem{PengTerng1983b} C.~K.~Peng and C.~L.~Terng, The scalar curvature of minimal hypersurfaces in spheres, \emph{Math. Ann.} \textbf{266} (1983), 105--113.

	\bibitem{ScherfnerWeissYau2012} M.~Scherfner, S.~Weiss, and S.-T.~Yau, A review of the Chern conjecture for isoparametric hypersurfaces in spheres, in \emph{Advances in Geometric Analysis}, Adv. Lect. Math. (ALM), vol.~21, International Press, Somerville, MA, 2012, pp.~175--187.

	\bibitem{Simons1968} J.~Simons, Minimal varieties in Riemannian manifolds, \emph{Ann. of Math. (2)} \textbf{88} (1968), 62--105.

	\bibitem{TangWeiYan2020} Z.~Z.~Tang, D.~Y.~Wei, and W.~J.~Yan, A sufficient condition for a hypersurface to be isoparametric, \emph{Tohoku Math. J. (2)} \textbf{72} (2020), no.~4, 493--505.

	\bibitem{TangYan2023} Z.~Z.~Tang and W.~J.~Yan, On the Chern conjecture for isoparametric hypersurfaces, \emph{Sci. China Math.} \textbf{66} (2023), 143--162.

    \bibitem{Valent1988} T.~Valent, Boundary Value Problems of Finite Elasticity: Local Theorems on Existence, Uniqueness, and Analytic Dependence on Data, \textit{Springer Tracts in Natural Philosophy}, vol. 31, Springer-Verlag, New York, 1988.

	\bibitem{White1991} B.~White, The space of minimal submanifolds for varying Riemannian metrics, \emph{Indiana Univ. Math. J.} \textbf{40} (1991), no.~1, 161--200.

	\bibitem{Wiygul2020} D.~Wiygul, Minimal surfaces in the 3-sphere by stacking Clifford tori, \emph{J. Differential Geom.} \textbf{114} (2020), no.~3, 467--549.

	\bibitem{XuXu2017} H.~W.~Xu and Z.~Y.~Xu, On Chern's conjecture for minimal hypersurfaces and rigidity of self-shrinkers, \emph{J. Funct. Anal.} \textbf{273} (2017), no.~11, 3406--3425.

    \bibitem{Yau1982} S.-T.~Yau, Problem section, in \emph{Seminar on Differential Geometry}, Ann. of Math. Stud., vol.~102, Princeton University Press, Princeton, NJ, 1982, pp.~669--706.

\end{thebibliography}
\end{document}